\documentclass[12pt]{article}

\usepackage[T1]{fontenc}
\usepackage[active]{srcltx}
\usepackage{lmodern}
\usepackage{amsmath,amstext,amssymb,graphicx,graphics,color}
\usepackage{mathtools,mathrsfs}
\usepackage{enumitem}
\usepackage{tikz}
\usepackage{mleftright}
\usepackage{pgfplots}
\usepackage{theorem}
\usepackage{cite}               
\usepackage{hyperref}           
\usepackage{bbm}                
\usepackage{fancybox}           
\usepackage[section]{placeins}  
\usepackage{mathrsfs}
\usepackage{subcaption}
\allowdisplaybreaks
\theorembodyfont{\normalfont\slshape} 
\newtheorem{definition}{Definition}
\newtheorem{theorem}{Theorem}
\newtheorem{proposition}{Proposition}[section]

\newtheorem{lemma}[proposition]{Lemma}
\theoremstyle{break} 

\newenvironment{proof}%
{{\par\noindent \bf Proof. \nobreak}}%
{\nobreak \removelastskip \nobreak \hfill $\Box$ \medbreak}
{{\par\noindent \bf Proof \nobreak}}%
{\nobreak \removelastskip \nobreak \hfill $\Box$ \medbreak}
{{\par\noindent \bf Proof lemma. \nobreak}}%
{\nobreak \removelastskip \nobreak \bf End proof lemma. \medbreak}
\newenvironment{remark}{\par \medskip \noindent {\bf Remark. }\nobreak}{\par \medskip}
\usepackage{fancyhdr}
 \fancypagestyle{plain}{
  \fancyhead{}
  
}

\def\paragraph#1{{\bf #1\ }}
\newcommand{\RN}[1]{%
  \textup{\uppercase\expandafter{\romannumeral#1}}%
}

\newcommand{\dd}{\mathrm{d}}

\newcommand{\HH}{\mathrm{H}}

\DeclareMathOperator*{\argmax}{\arg\!\max}
\usepackage[utf8]{inputenc}
\usepackage{unicode_character}

\title{Mean-field analysis of a random asset exchange model with probabilistic cheaters}

\author{Fei Cao \footnotemark[1]}

\begin{document}
\maketitle

\footnotetext[1]{University of Massachusetts Amherst - Department of Mathematics and Statistics, Amherst, MA 01003, USA}

\tableofcontents

\begin{abstract}
We investigate a variant of the standard Bennati-Dragulescu-Yakovenko (BDY) game \cite{dragulescu_statistical_2000} inspired by the very recent work \cite{blom_hallmarks_2024}, where agents involving in a money exchange dynamics are classified into two distinct types which are termed as probabilistic cheaters and honest players, respectively. A probabilistic cheater has a positive probability of declaring to have no money to give to other agents in the system, resulting in a potential financial benefits from being dishonest about his/her financial status. We provide a mean-field description of the agent-based model (in terms of a coupled infinite dimensional system of nonlinear ODEs), in the large population limit where the number of players is sent to infinity, and proves convergence of the coupled mean-field system to its stationary distribution (provided by a mixture of geometric distributions). In particular, the model gives rise to a novel formulation involving a mixture of probability distributions, thereby motivating the introduction of a unusual (generalized) entropy functional tailored to the associated mean-field system.
\end{abstract}


\section{Introduction}\label{sec:1}
\setcounter{equation}{0}

In this work, we analyze a modified version of the well-known Bennati-Dragulescu-Yakovenko (BDY) game \cite{dragulescu_statistical_2000} inspired by the recent joint work of Kristian Blom, Dmitrii Makarov and Alja{\v{z}} Godec \cite{blom_hallmarks_2024}. In the original BDY game, there are $N$ identity agents labelled by $1$ through $N$ and each of them is characterized by the amount of dollars he/she has. We denote by $S_i(t)$ the amount of dollars the agent $i$ has at time $t$. The game is a simple mechanism for dollar exchange taking place in a closed economical system, where at random times (generated by an exponential law) an agent $i$ picked uniformly at random gives a dollar (if he/she has at least one dollar) to another agent $j$ again picked uniformly at random, and if agent $i$ is ruined (i.e., $S_i= 0$) then nothing happens. We can represent the BDY game as follows:
\begin{equation}\label{dynamics:BDY}
\textbf{BDY game:} \qquad (S_i,S_j)~ \begin{tikzpicture} \draw [->,decorate,decoration={snake,amplitude=.4mm,segment length=2mm,post length=1mm}]
(0,0) -- (.6,0); \node[above,red] at (0.3,0) {};\end{tikzpicture}~  (S_i-1,S_j+1) \quad (\text{if } S_i\geq 1).
\end{equation}
Since the economical system is closed, we must have
\begin{equation}\label{eq:preserved_sum}
S_1(t)+ \cdots +S_N(t) = N\,\mu = \textrm{Constant} \qquad \text{for all } t\geq 0.
\end{equation}
for some $\mu > 0$. The BDY model described above is one of the earliest models in econophysics and has been studied extensively across different communities since its inception \cite{cao_derivation_2021,lanchier_rigorous_2017,merle_cutoff_2019}. Due to the fact that all agents with at least one dollar gives to the rest of agents at a fixed rate and the game is biased towards any specific agent (or certain group of agents), the BDY game is termed as the \emph{unbiased exchange model} in \cite{cao_derivation_2021,cao_interacting_2024} and the \emph{one-coin model} in \cite{lanchier_rigorous_2017}. Subsequent extensions of the basic BDY game suggest the presence of a bank (or even multiple banks) which allows agents to be indebted (i.e., $S_i < 0$), and we refer the interested readers to a series of recent works \cite{cao_bias_2023,cao_uncovering_2022,lanchier_rigorous_2019,lanchier_distribution_2022}. Other possible variations of the BDY mechanism also include the effect of bias, which amounts to adding tax or introducing a redistribution mechanism \cite{cao_uniform_2024,chakraborti_statistical_2000,chatterjee_pareto_2004,miao_nonequilibrium_2023}, that favors poorer agents or richer agents in each economical interaction between pair of agents. For an overview of the application to kinetic theory to econophysics and sociophysics models, we refer the interested readers to \cite{cao_fractal_2024,cao_iterative_2024,cao_k_2021,cao_sticky_2024,degond_continuum_2017,during_kinetic_2008,jabin_clustering_2014,loy_essentials_2025,matthes_steady_2008,pareschi_interacting_2013,pereira_econophysics_2017,savoiu_econophysics_2013}.

In this manuscript, we aim to analyze a variant of the model introduced in a recent work \cite{blom_hallmarks_2024}, which extends the classical BDY dynamics in the sense that the concept of probabilistic cheaters is introduced and therefore the model has two distinct types of agents. The settings of our model can be described as follows. Suppose that there are $N$ agents in total (labelled from $1$ to $N$) and each agent is either a honest player or a probabilistic cheater, denote by $n_h \in [0,1]$ and $n_c \in [0,1]$ the fraction of honest players and probabilistic cheaters in this economic system, respectively, with $n_h + n_c = 1$. We denote by $\mathcal{H}$ and $\mathcal{C}$ the collection of honest players and probabilistic cheaters, respectively. For the sake of concreteness and without loss of generality, we assume that $\mathcal{H} = \{1,\ldots,n_h\,N\}$ and $\mathcal{C} = \{1,\ldots,N\} \setminus \mathcal{H}$. At random times generated by an exponential law, we pick two different agents $i$ (``giver'') and $j$ (``receiver'') uniformly at random from $\{1,\ldots,N\}^2 \setminus \{i=j\}$ and update their wealth status according to the following rules:
\begin{enumerate}[label=(\roman*)]
  \item If $i \in \mathcal{H}$ (i.e., if agent $i$ is a honest player) and $S_i \geq 1$, then agent $i$ gives one dollar to agent $j$:
      \begin{equation}\label{dynamics:main}
(S_i,S_j)~ \begin{tikzpicture} \draw [->,decorate,decoration={snake,amplitude=.4mm,segment length=2mm,post length=1mm}]
(0,0) -- (.6,0); \node[above,red] at (0.3,0) {};\end{tikzpicture}~  (S_i-1,S_j+1);
\end{equation}
  \item If $i \in \mathcal{C}$ (i.e., if agent $i$ is a probabilistic cheater) and $S_i \geq 1$, then agent $i$ gives one dollar to agent $j$ with probability $1-\gamma$ for some fixed $\gamma \in (0,1)$, i.e., the transaction \eqref{dynamics:main} occurs with probability $1-\gamma$ and nothing happens with probability $\gamma$;
  \item If agent $i$ has no money in his/her pocket, then nothing happens.
\end{enumerate}
We remark here that the model parameter $\gamma \in [0,1)$ in the aforementioned binary dollar exchange dynamics can be interpreted as the probability to lie for a probabilistic cheater \cite{blom_hallmarks_2024}, since by pretending to have no money (with probability $\gamma$) a probabilistic cheater (with at least one dollar) can avoid transferring his/her wealth to other agents when he/she is picked to give out a dollar.
We also emphasize a major difference of our model from the model purposed in \cite{blom_hallmarks_2024}: it is assumed in \cite{blom_hallmarks_2024} that the exchange \eqref{dynamics:main} occurs (with probability $1$) whenever $S_j = 0$ and $S_i \geq 1$, i.e., in their model set-up a ``receiver'' with zero money will always receive a dollar from a ``giver'' as long as the ``giver'' (agent $i$) has at least one dollar, while for the model investigated in this manuscript, a probabilistic cheater will always give out a dollar with probability $1-\gamma$ regardless of the wealth status of the ``receiver'' (agent $j$).

We illustrate the dynamics in figure \ref{fig:illustration_model}. A central question in econophysics literature often revolves around identifying the limiting distribution of money among agents as the number of agents and the time horizon approach infinity. We demonstrate numerically the simulation results of our variant of the BDY dynamics involving probabilistic cheaters using $N=2000$ agents with the initial condition $S_i(0) = \mu$ for all $1\leq i \leq N$ in figure \ref{fig:agent_based_simulation}-left. Other model parameters employed in the agent-based simulation are $\mu =5$, $n_c = n_h =0.5$, $\gamma = 0.5$.

We observe in figure \ref{fig:agent_based_simulation_2} that the distributions of money among both the group $\mathcal{H}$ of honest players and the group $\mathcal{C}$ of probabilistic cheaters are approximately geometric (although with different parameters), as predicted by Proposition \ref{prop:equil}.

\begin{figure}[!htb]
\centering
\includegraphics[width = 0.6\textwidth]{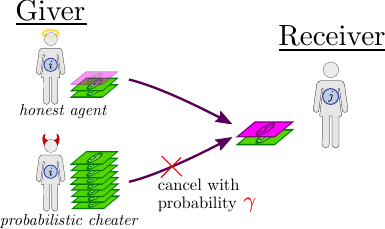}
\caption{Illustration of a variant of the BDY game involving probabilistic cheaters: at random times, a ``giver'' $i$ picked uniformly at random is selected to give one dollar to a ``receiver'' $j$ chosen uniformly at random as well. If agent $i$ has no dollar (i.e., $S_i = 0$) then nothing happens. Otherwise (i.e., $S_i\geq 1$), agent $i$ will give one dollar to agent $j$ if agent $i$ is an honest player (or equivalently $i \in \mathcal{H}$), and agent $i$ will give one dollar to agent $j$ with probability $1-\gamma$ if agent $i$ is a probabilistic cheater (i.e., $i \in \mathcal{C}$).}
\label{fig:illustration_model}
\end{figure}

\begin{figure}[!htb]
  \begin{subfigure}{0.47\textwidth}
    \centering
    \includegraphics[width=\textwidth,height=0.3\textheight]{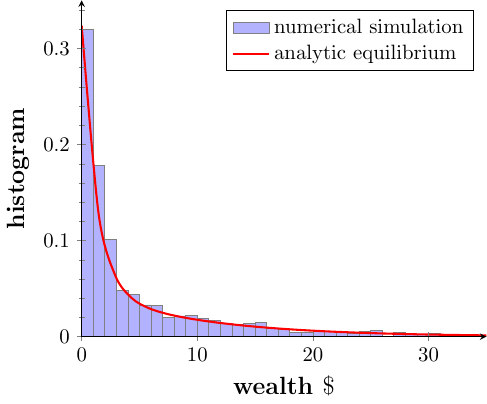}
  \end{subfigure}
  \hspace{0.1in}
  \begin{subfigure}{0.47\textwidth}
    \centering
    \includegraphics[width=\textwidth,height=0.3\textheight]{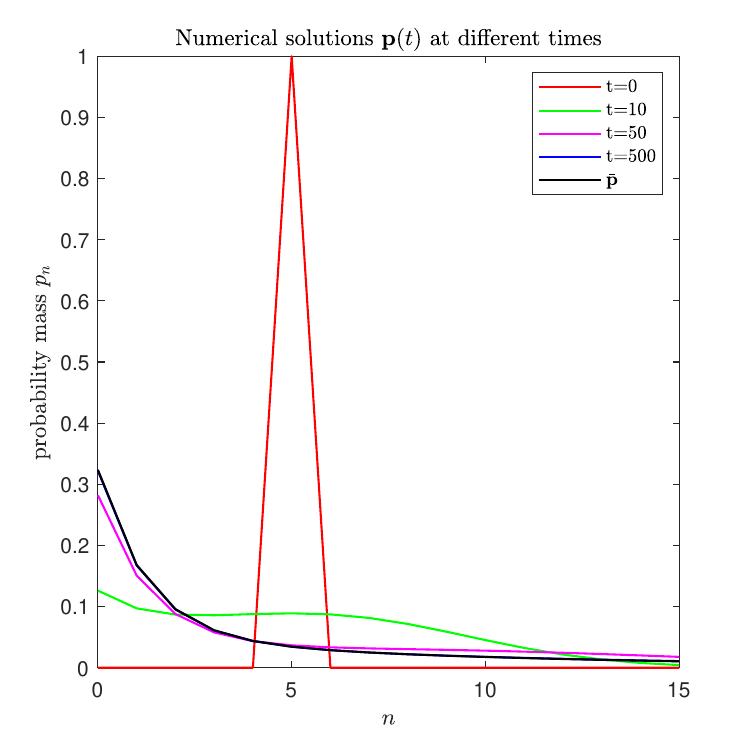}
  \end{subfigure}
  \caption{{\bf Left:} Distribution of money for the agent-based model with $N = 2,000$ agents after $20,000$ units of time, using $\mu =5$, $n_c = n_h =0.5$, $\gamma = 0.5$ and the initial condition $S_i(0) = \mu$ for all $1\leq i \leq N$. {\bf Right:} Simulation of the coupled mean-field ODE systems \eqref{eq:law_limit_honest} and \eqref{eq:law_limit_cheater} for $0\leq t\leq 500$, using $\mu =5$, $n_c = n_h =0.5$, $\gamma = 0.5$ and the initial condition ${\bf p}^h(0) = {\bf p}^c(0) = \delta_\mu$, where $\delta_\mu \in \mathcal{P}(\mathbb N)$ is Dirac-type distribution whose only non-zero component is located at its $(\mu+1)$-th coordinate. We observe that in both scenarios the terminal distributions are well-approximated by a convex combination of geometric distributions given by \eqref{eq:equilibria} below.}
  \label{fig:agent_based_simulation}
\end{figure}

\begin{figure}[!htb]
  \begin{subfigure}{0.47\textwidth}
    \centering
    \includegraphics[scale=0.8]{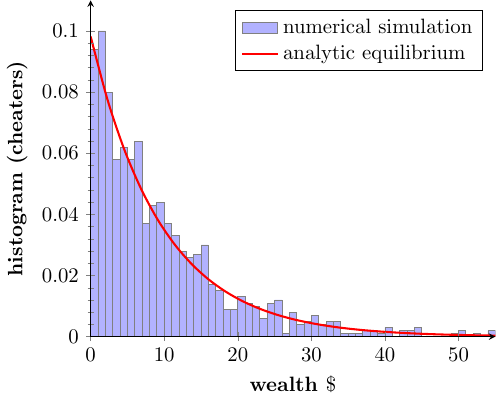}
  \end{subfigure}
  \hspace{0.1in}
  \begin{subfigure}{0.47\textwidth}
    \centering
    \includegraphics[scale=0.8]{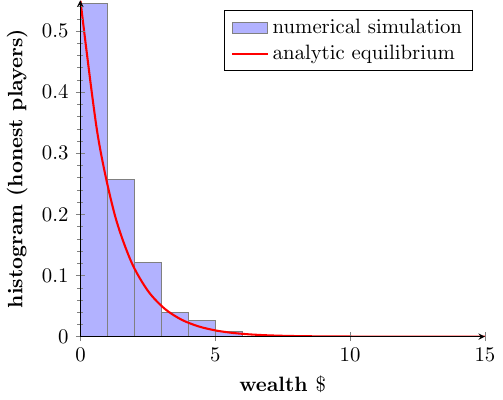}
  \end{subfigure}
  \caption{Distribution of money for the agent-based model with $N = 2,000$ agents after $20,000$ units of time, using $\mu =5$, $n_c = n_h =0.5$, $\gamma = 0.5$ and the initial condition $S_i(0) = \mu$ for all $1\leq i \leq N$. {\bf Left:} Wealth distribution among probabilistic cheaters.  {\bf Right:} Wealth distribution among honest players. We emphasize that for both sub-systems the large-time distributions are close to a geometric distribution given by ${\bf p}^c$ and ${\bf p}^h$ \eqref{eq:equilibria}, respectively.}
  \label{fig:agent_based_simulation_2}
\end{figure}

The remainder of this paper is organized as follows: in section \ref{sec:2} we provide the associated mean-field description of our agent-based dynamics, in the form a coupled infinite system of nonlinear ODEs \eqref{eq:law_limit_honest}-\eqref{eq:law_limit_cheater}, under the large population limit $N \to \infty$. We also explicitly identify the equilibrium distribution of money across the entire population, as well as within the sub-populations consisting solely of honest players and probabilistic cheaters, respectively. Section \ref{sec:3} is dedicated to the long-time analysis of the mean-field ODE system introduced in section \ref{sec:2}, with the primary objective of establishing its convergence to the corresponding equilibrium distribution. In particular, we introduce a novel generalized entropy functional tailored to the coupled mean-field system \eqref{eq:law_limit_honest}-\eqref{eq:law_limit_cheater}, and establish a non-trivial variational characterization of its equilibrium distribution (given by a convex combination of geometric distributions). In section \ref{sec:4} we briefly examine the impact of the cheating probability $\gamma$ on the wealth inequality of the equilibrium distribution (measured by Gini index), while keeping all other model parameters fixed. Finally, section \ref{sec:5} concludes the manuscript by outlining several potential directions for future research building upon the model introduced in this work.

\section{Mean-field system of nonlinear ODEs}\label{sec:2}
\setcounter{equation}{0}

At the mean-field level (as the number $N$ of agents goes to infinity), it is quite natural \cite{cao_bias_2023,cao_derivation_2021,cao_uncovering_2022} to describe the system via a probability mass function ${\bf p}(t)=\left(p_0(t),p_1(t),\ldots,p_n(t),\ldots\right)$, in which $p_n(t)$ represents the probability that a typical agent (i.e., an agent picked uniformly at random from the pool of $N$ agents) has $n$ dollars at time $t$. However, since in our model a generic agent has probability $n_h$ (probability $n_c = 1-n_h$, respectively) to be a honest player (a probabilistic cheater, respectively), a more convenient way of describing the system after passing to the large population limit $N \to \infty$ consists of a pair of distributions ${\bf p}^h(t) = \{p^h_n(t)\}_{n\geq 0}$ and ${\bf p}^c(t) = \{p^c_n(t)\}_{n\geq 0}$, where $p^h_n(t)$ ($p^c_n(t)$, respectively) denotes the proportion of agents among all honest players (among all probabilistic cheaters, respectively) having $n$ dollars at time $t$. In particular, we have
\begin{equation}\label{eq:p_vector}
{\bf p}(t) = n_c\cdot {\bf p}^c(t) + n_h\cdot {\bf p}^h(t).
\end{equation}
A standard mean-field type argument, similar to the one presented in \cite{blom_hallmarks_2024}, shows that the evolutions of ${\bf p}^h(t) = \{p^h_n(t)\}_{n\geq 0}$ and ${\bf p}^c(t) = \{p^c_n(t)\}_{n\geq 0}$ are described by the following coupled system of nonlinear ODEs:
\begin{equation}\label{eq:law_limit_honest}
\frac{\dd}{\dd t} {\bf p}^h(t) = \mathcal{L}^h[{\bf p}^h(t)],
\end{equation}
and
\begin{equation}\label{eq:law_limit_cheater}
\frac{\dd}{\dd t} {\bf p}^c(t) = \mathcal{L}^c[{\bf p}^c(t)],
\end{equation}
where
\begin{equation}\label{eq:Lh}
  \mathcal{L}^h[{\bf p}^h]_n \coloneqq \left\{
    \begin{array}{ll}
      p^h_1 - r\,p^h_0 &~ \text{for}~~ n = 0, \\
      p^h_{n+1} + r\,p^h_{n-1} - p^h_n - r\,p^h_n &~ \text{for}~~ n\geq 1,\\
    \end{array}
  \right.
\end{equation}
and
\begin{equation}\label{eq:Lc}
  \mathcal{L}^c[{\bf p}^c]_n \coloneqq \left\{
    \begin{array}{ll}
      (1-\gamma)\,p^c_1 - r\,p^c_0 &~ \text{for}~~ n = 0, \\
      (1-\gamma)\,p^c_{n+1} + r\,p^c_{n-1} - (1-\gamma)\,p^c_n - r\,p^c_n &~ \text{for}~~ n\geq 1,\\
    \end{array}
  \right.
\end{equation}
and
\begin{equation}\label{eq:def_r}
r \coloneqq n_c\,r_c\,(1-\gamma) + n_h\,r_h,~~r_h \coloneqq \sum\limits_{n\geq 1} p^h_n,~~r_c \coloneqq \sum\limits_{n\geq 1} p^c_n,
\end{equation}
in which $r_h$ and $r_c$ represent the proportion of agents who have at least one dollar among all honest players and among all probabilistic cheaters, respectively.

\begin{remark}
In the special case where $n_c = 0$ (and hence $n_h = 1$), i.e., when all agents in the system are honest players and no probabilistic cheaters are present, the ODE system \eqref{eq:law_limit_cheater} for probabilistic cheaters disappears and the mean-field dynamics \eqref{eq:law_limit_honest} for ${\bf p}^h(t)$ (or equivalently for ${\bf p}(t)$) coincides with the mean-field model corresponding to the classical BDY game \cite{cao_derivation_2021,dragulescu_statistical_2000}.
\end{remark}

In order to justify rigorously the passage from the stochastic $N$-agent model investigated in this work to the coupled infinite system of ODEs \eqref{eq:law_limit_honest} and \eqref{eq:law_limit_cheater} as $N \to \infty$, one need to prove the so-called \emph{propagation of chaos} \cite{chaintron_propagation_2022,chaintron_propagation_2022_partII,sznitman_topics_1991} which is beyond the scope of this manuscript. On the other hand, the propagation of chaos phenomenon has been proved rigorously for a number of econophysics models, and we refer the interested readers to \cite{cao_derivation_2021,cao_explicit_2021,cao_interacting_2024,cao_uniform_2024}.

Next, we collect some elementary observations regarding the solution of the ODE systems \eqref{eq:law_limit_honest} and \eqref{eq:law_limit_cheater}.
\begin{lemma}\label{lem:invariant}
Assume that ${\bf p}^h(t)=\{p^h_n(t)\}_{n\geq 0}$ and ${\bf p}^c(t) = \{p^c_n(t)\}_{n\geq 0}$ is a classical solution of \eqref{eq:law_limit_honest} and \eqref{eq:law_limit_cheater}, respectively, with ${\bf p}^c(0) \in \mathcal{P}(\mathbb N)$ and ${\bf p}^h(0) \in \mathcal{P}(\mathbb N)$ such that ${\bf p}(0) = n_c\cdot {\bf p}^c(0) + n_h\cdot {\bf p}^h(0)$ has mean $\mu$, then
\begin{equation}\label{eq:conservation_mass_mean_value}
\sum_{n=0}^\infty \mathcal{L}^h[{\bf p}^h]_n =0,\quad \sum_{n=0}^\infty \mathcal{L}^c[{\bf p}^c]_n =0,\quad \textrm{and} \quad \sum_{n=0}^\infty \left(n_h\,n\,\mathcal{L}^h[{\bf p}^h]_n + n_c\,n\,\mathcal{L}^c[{\bf p}^c]_n\right)=0.
\end{equation}
In particular, for all times $t\geq 0$ it holds that
\begin{equation}\label{eq:conservation_mass_mean_value_repeat}
\sum_{n=0}^\infty {\bf p}^h_n(t) = 1,\quad \sum_{n=0}^\infty {\bf p}^c_n(t) = 1,\quad \textrm{and} \quad n_h\,\sum_{n=0}^\infty n\,p^h_n(t) + n_c\,\sum_{n=0}^\infty n\,p^c_n(t) = \mu.
\end{equation}
\end{lemma}

The proof of Lemma \ref{lem:invariant} follows from straightforward computations and hence will be omitted. We only emphasize here that one can interpret $\sum_{n=0}^\infty n\,p^h_n$ and $\sum_{n=0}^\infty n\,p^c_n$ as the average amount of dollars that a typical honest player and a typical probabilistic cheater has, respectively.

Thanks to Lemma \ref{lem:invariant}, we deduce that ${\bf p}(t) \in V_\mu$ for all $t\geq 0$, where
\begin{equation}\label{eq:space}
V_\mu \coloneqq \left\{{\bf p} \mid {\bf p}=n_c\cdot {\bf p}^c + n_h\cdot {\bf p}^h;~ {\bf p}^h, {\bf p}^c \in \mathcal{P}(\mathbb N);~  n_h\,\sum_{n=0}^\infty n\,p^h_n + n_c\,\sum_{n=0}^\infty n\,p^c_n = \mu\right\}.
\end{equation}
We are now ready to identify the equilibrium distribution associated with the dynamics \eqref{eq:law_limit_honest} and \eqref{eq:law_limit_cheater}, respectively.

\begin{proposition}\label{prop:equil}
Under the settings of Lemma \ref{lem:invariant}, denote by $\bar{{\bf p}}^h$ and $\bar{{\bf p}}^c$ the equilibrium solution to \eqref{eq:law_limit_honest} and \eqref{eq:law_limit_cheater}, respectively. Then for all $n \in \mathbb N$,
\begin{equation}\label{eq:equilibria}
\bar{p}^h_n = \bar{p}^h_0\,\bar{r}^n \quad \textrm{and} \quad \bar{p}^c_n = \bar{p}^c_0\,\left(\frac{\bar{r}}{1-\gamma}\right)^n,
\end{equation}
where $\bar{p}^h_0 = 1 - \bar{r}$, $\bar{p}^c_0 = 1 - \frac{\bar{r}}{1-\gamma}$, and
\begin{equation}\label{eq:r_bar}
\bar{r} = \frac{1}{2\,(\mu+1)}\,\left[(2-\gamma)\,\mu + (1-\gamma\,n_h) - \sqrt{((2-\gamma)\,\mu + (1-\gamma\,n_h))^2 - 4\,(1-\gamma)\,(\mu+1)\,\mu}\right]
\end{equation}
is the unique solution within $(0,1-\gamma)$ which satisfies
\begin{equation}\label{eq:quadratic}
(\mu+1)\,\bar{r}^2 - \left[(2-\gamma)\,\mu + (1-\gamma\,n_h)\right]\,\bar{r} + (1-\gamma)\,\mu = 0.
\end{equation}
\end{proposition}

\begin{proof}
Setting $\mathcal{L}^h[{\bf p}^h]_n = 0$ and $\mathcal{L}^c[{\bf p}^c]_n = 0$ for each $n\in \mathbb N$ leads us to \eqref{eq:equilibria}. As $\bar{{\bf p}}^h \in \mathcal{P}(\mathbb N)$ and $\bar{{\bf p}}^c \in \mathcal{P}(\mathbb N)$, we deduce that $\bar{p}^h_0 = 1 - \bar{r}$ and $\bar{p}^c_0 = 1 - \frac{\bar{r}}{1-\gamma}$. Since $\bar{{\bf p}} \coloneqq n_c\cdot \bar{{\bf p}}^c + n_h\cdot \bar{{\bf p}}^h$ lives in $V_\mu$ \eqref{eq:space}, we also have
\begin{equation*}
\mu = n_c\,\frac{\bar{r}}{1-\gamma - \bar{r}} + n_h\,\frac{\bar{r}}{1-\bar{r}} = (1-n_h)\,\frac{\bar{r}}{1-\gamma - \bar{r}} + n_h\,\frac{\bar{r}}{1-\bar{r}},
\end{equation*}
which is equivalent to \eqref{eq:quadratic} after simple rearrangements. To finish the proof it remains to show that the quadratic equation in $\bar{r}$ has two real and distinct roots (and as an easy consequence of Vieta's formulas, the equation \eqref{eq:quadratic} admits a unique solution within $(0,1-\gamma)$). To this end, it suffices to notice that
\begin{align*}
(2-\gamma)\,\mu + (1-\gamma\,n_h) &> \mu + (1-\gamma)\,\mu + (1-\gamma) \\
&= \mu + (1-\gamma)\,(\mu + 1) \geq 2\,\sqrt{\mu+1}\,\sqrt{(1-\gamma)\,\mu}.
\end{align*}
\end{proof}

\begin{remark}
Under the settings of the previous remark, i.e., in the special case when $(n_c,n_h) = (0,1)$, the equilibrium distribution $\bar{{\bf p}}^h$ (or equivalently $\bar{{\bf p}}$) simplifies to a geometric distribution with mean $\mu$ (i.e., $\bar{p}_n = \frac{1}{1+\mu}\left(\frac{\mu}{1+\mu}\right)^n$ for all $n\in \mathbb N$), which corresponds to the equilibrium distribution for the standard BDY model \cite{cao_derivation_2021,dragulescu_statistical_2000}.
\end{remark}

\begin{remark}
At equilibrium, since
\begin{equation}\label{eq:observation}
\sum_{n\geq 0} n\,\bar{{\bf p}}^h_n = \frac{\bar{r}}{1-\bar{r}} < \sum_{n\geq 0} n\,\bar{{\bf p}}^c_n = \frac{\bar{r}}{1-\gamma - \bar{r}},
\end{equation}
we thus conclude that probabilistic cheaters have a ``collective advantage'' in the sense that, on average, a typical probabilistic cheater is wealthier than a typical honest player at the mean-field level when $t \to \infty$.
\end{remark}

\section{Large time behavior}\label{sec:3}
\setcounter{equation}{0}

\subsection{Convergence to mixtures of geometric distributions}\label{subsec:3.1}
In this subsection, we focus on the problem of showing large time convergence of solutions of the coupled nonlinear ODE systems \eqref{eq:law_limit_honest} and \eqref{eq:law_limit_cheater} to their unique equilibrium solutions $\bar{{\bf p}}^h$ and $\bar{{\bf p}}^c$, respectively. For this purpose, we naturally search for certain energy functionals which are monotonic when evaluated along the evolution equations \eqref{eq:law_limit_honest} and \eqref{eq:law_limit_cheater}. In other words, a key ingredient which we will rely on is the construction of an appropriate Lyapunov functional associated to the evolution systems \eqref{eq:law_limit_honest} and \eqref{eq:law_limit_cheater}. As it turns out, the following energy functional, defined for a pair of distributions $({\bf f},{\bf g}) \in \mathcal{P}(\mathbb N) \times \mathcal{P}(\mathbb N)$, is monotonically increasing along classical solutions of \eqref{eq:law_limit_honest} and \eqref{eq:law_limit_cheater}:
\begin{equation}\label{eq:energy}
\HH[({\bf f},{\bf g})] \coloneqq -n_c\,\sum_{n\geq 0} f_n\,\log f_n - n_h\,\sum_{n\geq 0} g_n\,\log g_n - n_c\,\log(1-\gamma)\,\sum_{n\geq 0} n\,f_n.
\end{equation}
In other words, we claim that $\HH[({\bf p}^c(t), {\bf p}^h(t))]$ increases as time $t$ increases. The underlying principle which leads us to the construction of the above generalized $\HH$ functional is encoded in the following non-trivial variational charaterization of the joint geometric distribution pair $(\bar{{\bf p}}^c, \bar{{\bf p}}^h)$, which is of independent interest:

\begin{lemma}[Variational characterization of the geometric pair $(\bar{{\bf p}}^c, \bar{{\bf p}}^h)$]
\label{lem:max_entropy_characterization}
Let \[\mathcal{A}_\mu = \left\{({\bf f},{\bf g}) \in \mathcal{P}(\mathbb N) \times \mathcal{P}(\mathbb N) \mid n_c\,\sum_{n\geq 0} n\,f_n + n_h\,\sum_{n\geq 0} n\,g_n = \mu \right\},\] then we have
\begin{equation}\label{eq:max_entropy_characterization}
(\bar{{\bf p}}^c, \bar{{\bf p}}^h) = \argmax_{({\bf f},{\bf g}) \in \mathcal{A}_\mu} \HH[({\bf f},{\bf g})].
\end{equation}
\end{lemma}

\begin{proof}
The proof of Lemma \ref{lem:max_entropy_characterization} relies primarily on a standard application of the method of Lagrange multipliers so we just present a sketch of the proof. Let us introduce the relevant Lagrangian defined by
\begin{equation*}
\begin{aligned}
\mathcal{L}[({\bf f},{\bf g})] &= \HH[({\bf f},{\bf g})] - \lambda_1\,\left(\sum_{n\geq 0} f_n - 1\right) - \lambda_2\,\left(\sum_{n\geq 0} g_n - 1\right) \\
&\qquad - \lambda_3\,\left(n_c\,\sum_{n\geq 0} n\,f_n + n_h\,\sum_{n\geq 0} n\,g_n - \mu\right),
\end{aligned}
\end{equation*}
in which $\lambda_1$, $\lambda_2$ and $\lambda_3$ are Lagrange multipliers. By setting $\frac{\partial \mathcal{L}}{\partial f_n} = 0 $, $\frac{\partial \mathcal{L}}{\partial g_n} = 0$ for all $n \in \mathbb N$, and $\frac{\partial \mathcal{L}}{\partial \lambda_i} = 0$ for $1\leq i\leq 3$ lead us to the advertised conclusion \eqref{eq:max_entropy_characterization}.
\end{proof}

\begin{remark}
In the special case where $n_c = 0$ and $n_h = 1$, the functional $\HH[({\bf f},{\bf g})]$ boils down to the (Shannon) entropy of the law ${\bf g} \in \mathcal{P}(\mathbb N)$, whence the content of Lemma \ref{lem:max_entropy_characterization} simplifies to the well-known maximum entropy characterization of the geometric distribution \cite{cover_elements_1999} (i.e., the geometric distribution with mean $\mu$, defined by $\bar{p}_n = \frac{1}{1+\mu}\left(\frac{\mu}{1+\mu}\right)^n$ for all $n\in \mathbb N$, maximizes the entropy over all probabilities on $\mathbb N$ with a prescribed mean value $\mu \in \mathbb{R}_+$). Consequently, we may view Lemma \ref{lem:max_entropy_characterization} as the maximum entropy characterization of the joint geometric distributions $(\bar{{\bf p}}^c, \bar{{\bf p}}^h)$, which generalizes the standard maximum entropy characterization of a geometric distribution.
\end{remark}

We are now ready to prove the main result in this section, regarding the monotonicity (with respect to time $t$) of the generalized $\HH$ functional $\HH[({\bf p}^c(t), {\bf p}^h(t))]$ along the evolution systems \eqref{eq:law_limit_honest} and \eqref{eq:law_limit_cheater}. The production of an appropriate entropy functional along the solution trajectory of certain Boltzmann-type (kinetic) equations is clearly reminiscent of the celebrated Boltzmann's $\HH$ theorem from classical statistical physics \cite{villani_review_2002,villani_entropy_2006}.

\begin{theorem}[Entropy production and convergence to equilibrium]
\label{thm:main}
Under the settings of Lemma \ref{lem:invariant}, for all $t\geq 0$ it holds that
\begin{equation}\label{eq:dissipation_E}
\begin{aligned}
0&\leq \frac{\dd}{\dd t} \HH[({\bf p}^c, {\bf p}^h)] \\
&= (1-\gamma)\,n_c\,\sum_{n\geq 0}\left(p^c_{n+1}-\frac{r\,p^c_n}{1-\gamma}\right)\log\frac{p^c_{n+1}}{r\,p^c_n/(1-\gamma)} + n_h\,\sum_{n\geq 0}\left(p^h_{n+1}- r\,p^h_n\right)\log\frac{p^h_{n+1}}{r\,p^h_n}.
\end{aligned}
\end{equation}
Moreover, $\frac{\dd}{\dd t} \HH[({\bf p}^c(t), {\bf p}^h(t))] = 0$ if and only if the pair $({\bf p}^c,{\bf p}^h)$ coincides with $(\bar{{\bf p}}^c,\bar{{\bf p}}^h)$ \eqref{eq:equilibria}. Consequently, ${\bf p}^c(t) \xrightarrow{t\to \infty} \bar{{\bf p}}^c$ and ${\bf p}^h(t) \xrightarrow{t\to \infty} \bar{{\bf p}}^h$ component-wisely.
\end{theorem}

\begin{proof}
We evaluate the time derivative of each of the terms appearing in the definition of $\HH[({\bf p}^c(t), {\bf p}^h(t))]$ separately. Straightforward computations yield that
\begin{align*}
\frac{\dd}{\dd t} \sum_{n\geq 0} p^h_n\,\log p^h_n &= \sum_{n\geq 0} \left[p^h_{n+1} + r\,\mathbbm{1}\{n\geq 1\}\,p^h_{n-1} - \mathbbm{1}\{n\geq 1\}\,p^h_n - r\,p^h_n\right]\,\log p^h_n \\
&= \sum_{n\geq 0} \left[r\,p^h_n - p^h_{n+1}\right]\,\log \frac{p^h_{n+1}}{p^h_n} \\
&= -\sum_{n\geq 0} \left[p^h_{n+1} - r\,p^h_n\right]\,\log \frac{p^h_{n+1}}{r\,p^h_n} + (r-r_h)\,\log r
\end{align*}
and in a similar fashion that
\begin{align*}
\frac{\dd}{\dd t} \sum_{n\geq 0} p^c_n\,\log p^c_n &= \sum_{n\geq 0} \left[(1-\gamma)\,p^c_{n+1} + r\,\mathbbm{1}\{n\geq 1\}\,p^c_{n-1} - \mathbbm{1}\{n\geq 1\}\,(1-\gamma)\,p^c_n - r\,p^c_n\right]\,\log p^c_n \\
&= (1-\gamma)\,\sum_{n\geq 0} \left[\frac{r\,p^c_n}{1-\gamma} - p^c_{n+1}\right]\,\log \frac{p^c_{n+1}}{p^c_n} \\
&= -(1-\gamma)\,\sum_{n\geq 0} \left[p^c_{n+1} - \frac{r\,p^c_n}{1-\gamma}\right]\,\log \frac{p^c_{n+1}}{\frac{r\,p^c_n}{1-\gamma}} + \left(r-(1-\gamma)\,r_c\right)\,\log \frac{r}{1-\gamma}.
\end{align*}
Moreover, we can readily verify that
\begin{equation*}
\frac{\dd}{\dd t} \sum_{n\geq 0} n\,p^h_n = r - r_h \quad \textrm{and} \quad \frac{\dd}{\dd t} \sum_{n\geq 0} n\,p^c_n = r - (1-\gamma)\,r_c.
\end{equation*}
Assembling these identities lead us to the announced result \eqref{eq:dissipation_E}. Now thanks to \eqref{eq:dissipation_E}, we deduce that $\frac{\dd}{\dd t} \HH[({\bf p}^c(t), {\bf p}^h(t))] = 0$ if and only if $p^c_{n+1} = \frac{r\,p^c_n}{1-\gamma}$ and $p^h_{n+1} = r\,p^h_n$ for all $n\in \mathbb N$, or equivalently $({\bf p}^c, {\bf p}^h) = (\bar{{\bf p}}^c,\bar{{\bf p}}^h)$ \eqref{eq:equilibria}. Finally, the large time convergence guarantees ${\bf p}^c(t) \xrightarrow{t\to \infty} \bar{{\bf p}}^c$ and ${\bf p}^h(t) \xrightarrow{t\to \infty} \bar{{\bf p}}^h$ stated in Theorem \ref{thm:main} can be established in a similar way as for the standard mead-field BDY dynamics where $n_c = 0$ (i.e., without the presence of probabilistic cheaters), see for instance \cite{merle_cutoff_2019,cao_derivation_2021}.
\end{proof}

We demonstrate numerically the convergence of ${\bf p}(t)$ to its equilibrium distribution $\bar{{\bf p}} = n_c\cdot \bar{{\bf p}}^c + n_h\cdot \bar{{\bf p}}^h$ in figure \ref{fig:agent_based_simulation}-right, with the following model parameters: $\mu =5$, $n_c = n_h =0.5$, $\gamma = 0.5$. To discretize the ODE system, we employ $500$ components to describe the distribution ${\bf p}(t)$ (i.e., ${\bf p}(t) \approx (p_0(t),\ldots,p_{500}(t))$). As initial condition, we use $p_{\mu}(0)= 1$ and $p_n(0) = 0$ for all $n \neq \mu$. The standard Runge-Kutta fourth-order scheme is used to discretize the ODE systems \eqref{eq:law_limit_honest} and \eqref{eq:law_limit_cheater} with the time step $\Delta t=0.01$.
We plot in figure \ref{fig:ODE_simulation}-left and figure \ref{fig:ODE_simulation}-right the evolution of the numerical solutions ${\bf p}^c$ and ${\bf p}^h$ at different times for $0\leq t\leq 500$, respectively. It can be observed that convergence to a geometric distribution occur in each of the sub-systems.

\begin{figure}[!htb]
  \begin{subfigure}{0.47\textwidth}
    \centering
    \includegraphics[scale=0.6]{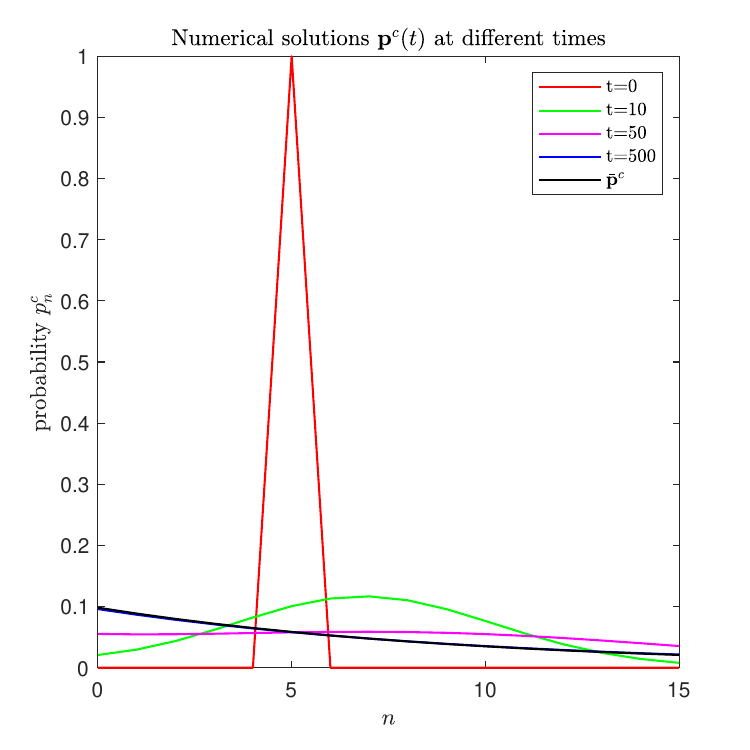}
  \end{subfigure}
  \hspace{0.1in}
  \begin{subfigure}{0.47\textwidth}
    \centering
    \includegraphics[scale=0.6]{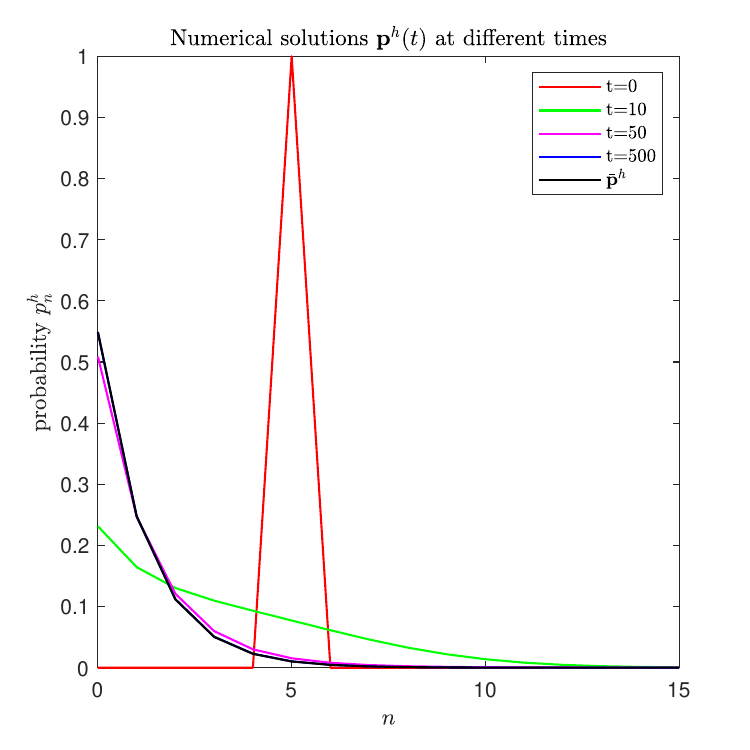}
  \end{subfigure}
  \caption{{\bf Left}: Evolution of the numerical solution ${\bf p}^c(t)$ of the ODE system \eqref{eq:law_limit_cheater} at various times. {\bf Right}: Evolution of the numerical solution ${\bf p}^h(t)$ of the ODE system \eqref{eq:law_limit_honest} at various times. For both sub-systems, ${\bf p}^c(t=500)$ and ${\bf p}^h(t=500)$ are almost indistinguishable from their respective geometric equilibrium distributions $\bar{{\bf p}}^c$ and $\bar{{\bf p}}^h$.}
  \label{fig:ODE_simulation}
\end{figure}

To illustrate the production of the entropy-like $\HH$ functional $\HH[({\bf p}^c, {\bf p}^h)]$ numerically, we use the same set-up as in the previous experiment shown in figure \ref{fig:ODE_simulation}. Figure \ref{fig:evolution_H}-left shows the evolution of $\HH[({\bf p}^c, {\bf p}^h)]$ over $0\leq t \leq 500$ and figure \ref{fig:evolution_H}-right shows the evolution of $\HH[(\bar{{\bf p}}^c, \bar{{\bf p}}^h)] - \HH[({\bf p}^c, {\bf p}^h)]$, both of which provide numerical evidence towards the monotonicity and convergence of the $\HH$ functional towards its equilibrium value.

\begin{figure}[!htb]
  \begin{subfigure}{0.47\textwidth}
    \centering
    \includegraphics[scale=0.6]{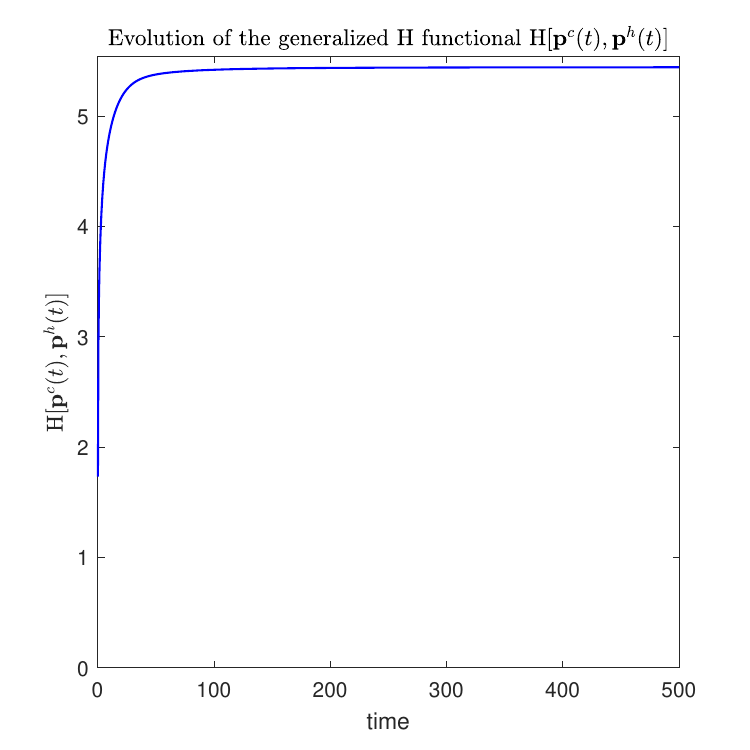}
  \end{subfigure}
  \hspace{0.1in}
  \begin{subfigure}{0.47\textwidth}
    \centering
    \includegraphics[scale=0.6]{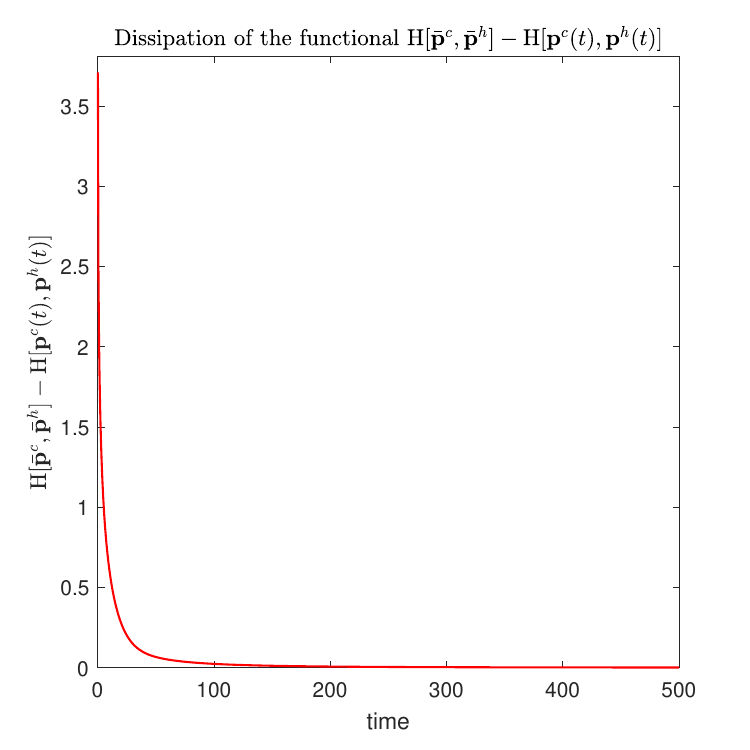}
  \end{subfigure}
  \caption{{\bf Left}: Evolution of the entropy-like $\HH$ functional $\HH[({\bf p}^c, {\bf p}^h)]$ over $0\leq t \leq 500$. {\bf Right}: Evolution of $\HH[(\bar{{\bf p}}^c, \bar{{\bf p}}^h)] - \HH[({\bf p}^c, {\bf p}^h)]$ over $0\leq t \leq 500$.}
  \label{fig:evolution_H}
\end{figure}

\subsection{Linearization around the equilibrium}\label{subsec:3.2}


In this section we intend to carry out a standard linearization analysis near the equilibrium distribution given by the geometric pair $(\bar{{\bf p}}^c, \bar{{\bf p}}^h)$. For this purpose, we linearize the systems \eqref{eq:law_limit_cheater} and \eqref{eq:law_limit_honest} around $\bar{{\bf p}}^c$ and $\bar{{\bf p}}^h$ respectively, by setting $p^c_n = \bar{p}^c_n + \varepsilon\,w^c_n$ and $p^h_n = \bar{p}^h_n + \varepsilon\,w^h_n$ for $0< |\varepsilon| \ll 1$ and for all $n \in \mathbb N$. This leads to the following coupled systems of linearized ODEs satisfied by ${\bf w} = ({\bf w}^c, {\bf w}^h)$:
\begin{equation}\label{eq:w^c}
\left\{
\begin{array}{ll}
\frac{\dd}{\dd t} w^c_0 = (1-\gamma)\,w^c_1 - \bar{r}\,w^c_0 - r_{{\bf w}}\,\bar{p}^c_0, & \\
\frac{\dd}{\dd t} w^c_n = (1-\gamma)\,w^c_{n+1} + \bar{r}\,w^c_{n-1} - (1-\gamma)\,w^c_n - \bar{r}\,w^c_n + r_{{\bf w}}\,(\bar{p}^c_{n-1}-\bar{p}^c_n), &~ \text{for}~~ n\geq 1\\
    \end{array}
  \right.
\end{equation}
and
\begin{equation}\label{eq:w^h}
\left\{
\begin{array}{ll}
\frac{\dd}{\dd t} w^h_0 = w^h_1 - \bar{r}\,w^h_0 - r_{{\bf w}}\,\bar{p}^h_0, & \\
\frac{\dd}{\dd t} w^h_n = w^h_{n+1} + \bar{r}\,w^h_{n-1} - w^h_n - \bar{r}\,w^h_n + r_{{\bf w}}\,(\bar{p}^h_{n-1}-\bar{p}^h_n), &~ \text{for}~~ n\geq 1\\
    \end{array}
  \right.
\end{equation}
in which $r_{{\bf w}} \colon = n_c\,(1-\gamma)\,\sum_{n\geq 1} w^c_n + n_h\,\sum_{n\geq 1} w^h_n$. Meanwhile, it is straightforward to check that the (admissible perturbation) pair $({\bf w}^c, {\bf w}^h)$ satisfies $({\bf w}^c(t), {\bf w}^h(t)) \in \mathcal{A}$ where
\[\mathcal{A} = \left\{({\bf w}^c,{\bf w}^h) \in \ell^1(\mathbb N) \times \ell^1(\mathbb N) \mid \sum_{n\geq 0} w^c_n = 0; \sum_{n\geq 0} w^h_n = 0; n_h\,\sum_{n\geq 0} n\,w^h_n + n_c\,\sum_{n\geq 0} n\,w^c_n = 0\right\}\]
for all $t \geq 0$, thus we also have $r_{{\bf w}} = -n_c\,(1-\gamma)\,w^c_0 - n_h\,w^h_0$. Linearizing the $\HH$ functional \eqref{eq:energy} (or more precisely the functional $-\HH$) suggests the following candidate Lyapunov functional associated with the coupled linear systems \eqref{eq:w^c}-\eqref{eq:w^h}:
\begin{equation}\label{eq:energy_E}
\mathcal{E}\left[({\bf w}^c, {\bf w}^h)\right] \coloneqq n_c\,\sum_{n\geq 0} \frac{(w^c_n)^2}{\bar{p}^c_n} + n_h\,\sum_{n\geq 0} \frac{(w^h_n)^2}{\bar{p}^h_n}.
\end{equation}
Our main result in this section is to establish the (monotone in time) dissipation of the linearized energy functional $\mathcal{E}$ \eqref{eq:energy_E} along solutions of \eqref{eq:w^c}-\eqref{eq:w^h}. We start with the following generic observation related to a geometric distribution.

\begin{lemma}\label{lem:elementary}
Assume that ${\bf p} = (p_0,p_1,\ldots)$ is a geometric distribution with $p_n = (1-r)\,r^n$ for all $n \in \mathbb N$ where $r \in (0,1)$. Then for any ${\bf y} \in \ell^1(\mathbb N)$ which satisfies $\sum_{n\geq 0} y_n = 0$ and $\sum_{n\geq 0} n\,y_n = 0$, it holds that
\begin{equation*}
y^2_0 \leq (1-r)\,r^2\,\sum_{n\geq 0} \frac{y^2_n}{p_n}
\end{equation*}
\end{lemma}

\begin{proof}
Let $\lambda \in \mathbb R$ and write $y_0 = \sum_{n\geq 0} (\lambda\,n-1)\,y_n$. Denote by $m = \sum_{n\geq 1} n\,p_n = \frac{r}{1-r}$ the mean value of the geometric distribution ${\bf p}$, the classical Cauchy-Schwarz inequality leads us to
\begin{equation*}
y^2_0 \leq \sum_{n\geq 1} (\lambda\,n-1)^2\,p_n\,\sum_{n\geq 1} \frac{y^2_n}{p_n} = \left(\lambda^2\,(m + 2\,m^2) - 2\,\lambda\,m + r\right)\sum_{n\geq 1} \frac{y^2_n}{p_n}.
\end{equation*}
Optimizing the choice of $\lambda$ yields that $\lambda = \frac{1}{1+2m}$, which gives rise to the advertised bound.
\end{proof}

\begin{proposition}\label{prop:energy_dissipation}
Assume that ${\bf w} = ({\bf w}^c, {\bf w}^h)$ is a solution of the linear ODE systems \eqref{eq:w^c}-\eqref{eq:w^h}, then for all $t\geq 0$ we have
\begin{equation*}
\frac{\dd}{\dd t} \mathcal{E}\left[({\bf w}^c, {\bf w}^h)\right] \leq 0.
\end{equation*}
\end{proposition}

\begin{proof}
After lengthy but routine computations, we obtain
\begin{equation}\label{eq:roc}
\begin{aligned}
\frac{\dd}{\dd t} \mathcal{E}\left[({\bf w}^c, {\bf w}^h)\right] &= -n_h\,\sum_{n\geq 0} \frac{\left(w^h_{n+1}-w^h_n\right)^2}{\bar{p}^h_n} - n_c\,(1-\gamma)\,\sum_{n\geq 0} \frac{\left(w^c_{n+1}-w^c_n\right)^2}{\bar{p}^c_n} \\
&\qquad + n_h\,(w^h_0)^2 + n_c\,(1-\gamma)\,(w^c_0)^2 + \frac{r^2_{{\bf w}}}{\bar{r}}.
\end{aligned}
\end{equation}
Since $r^2_{{\bf w}} = \left(n_c\,(1-\gamma)\,w^c_0 + n_h\,w^h_0\right)^2 \leq n_c\,(1-\gamma)^2\,(w^c_0)^2 + n_h\,(w^h_0)^2$, we deduce that
\begin{equation}\label{eq:i0}
\begin{aligned}
\frac{\dd}{\dd t} \mathcal{E}\left[({\bf w}^c, {\bf w}^h)\right] &\leq -n_h\,\sum_{n\geq 0} \frac{\left(w^h_{n+1}-w^h_n\right)^2}{\bar{p}^h_n} - n_c\,(1-\gamma)\,\sum_{n\geq 0} \frac{\left(w^c_{n+1}-w^c_n\right)^2}{\bar{p}^c_n} \\
&\qquad + n_h\,\frac{1+\bar{r}}{\bar{r}}\,(w^h_0)^2 + n_c\,(1-\gamma)\,\frac{1-\gamma+\bar{r}}{\bar{r}}\,(w^c_0)^2.
\end{aligned}
\end{equation}
Invoking Lemma \eqref{lem:elementary} with $y_0 = w^h_0$, $y_n = w_n - w_{n-1}$ for $n\geq 1$, and $p_n = \bar{p}^h_n$ gives rise to
\[(w^h_0)^2 \leq (1-\bar{r})\,\bar{r}^2\,\left[\frac{(w^h_0)^2}{\bar{p}^h_n} + \sum_{n\geq 1} \frac{\left(w^h_n-w^h_{n-1}\right)^2}{\bar{p}^h_n}\right],\] or equivalently
\begin{equation}\label{eq:i1}
(w^h_0)^2 \leq \frac{\bar{r}}{1+\bar{r}}\,\sum_{n\geq 1} \frac{\left(w^h_n-w^h_{n-1}\right)^2}{\bar{p}^h_{n-1}}.
\end{equation}
A similar consideration also yields the estimate
\begin{equation}\label{eq:i2}
(w^c_0)^2 \leq \frac{\bar{r}}{1-\gamma+\bar{r}}\,\sum_{n\geq 1} \frac{\left(w^c_n-w^c_{n-1}\right)^2}{\bar{p}^c_{n-1}}.
\end{equation}
Finally, inserting the estimates \eqref{eq:i1} and \eqref{eq:i2} into \eqref{eq:i0} allows us to conclude the proof.
\end{proof}

\begin{remark}
We speculate that the linearized energy $\mathcal{E}\left[({\bf w}^c, {\bf w}^h)\right]$ might decay exponentially fast in time, at least for certain choices of the model parameters (i.e., $n_h$, $\gamma$ and $\mu$). Indeed, in the special case when $n_h = 1$ (or equivalently when $n_c = 0$), quantitative exponential decay of the (linearized) energy $\mathcal{E}$ can be established for small enough $\mu$ thanks to a weighted Poincar\'e-type inequality satisfied by the geometric distribution ${\bf p}$ \cite{cao_uncovering_2022}. However, under the apparently more sophisticated settings considered in this work, we fail to obtain a quantitative (exponentially fast in time) convergence guarantee for the energy $\mathcal{E}\left[({\bf w}^c, {\bf w}^h)\right]$ towards zero, partly due to the potential lack of a weighted Poincar\'e-type for the convex combination of two geometric distributions $\bar{p} = n_h\,\bar{p}^h + n_c\,\bar{p}^c$.
\end{remark}

\section{Fraud induced wealth inequality}\label{sec:4}
\setcounter{equation}{0}

In this section, we would like to explore the presence of probabilistic cheaters on the inequality (measured by the well-known Gini index) of the equilibrium wealth distribution. To be more precise, we intend to investigate the Gini index of equilibrium distribution $\bar{{\bf p}}$ as a function of $\gamma \in [0,1)$ (probability of cheating for a probabilistic cheater) when other model parameters (i.e., $\mu$, $n_h$ and $n_c$) are frozen. For this purpose we first recall the definition of Gini index:

\begin{definition}[\textbf{Gini index}]
For a probability distribution ${\bf p} \in \mathcal{P}(\mathbb N)$ with a finite mean value $\mu \in \mathbb{R}_+$, the Gini index of ${\bf p}$ is defined via
\begin{equation}\label{def1:Gini}
G[{\bf p}] = \frac{1}{2\,\mu} \sum\limits_{i\in \mathbb N}\sum\limits_{j \in \mathbb N} |i-j|\,p_i\,p_j.
\end{equation}
\end{definition}

The Gini index $G$ is a widely recognized measure of (wealth) inequality which quantifies the disparity in a one-dimensional probability distribution. It ranges from 0 (representing perfect equality) to 1 (indicating extreme inequality). We now provide an explicit expression for the Gini index $G[\bar{{\bf p}}]$ of the equilibrium distribution.

\begin{proposition}\label{prop:Gini_equili}
For $\bar{{\bf p}} = n_c\,\bar{{\bf p}}^c + n_h\,\bar{{\bf p}}^h$, we have
\begin{equation}\label{eq:Gini_equili}
G[\bar{{\bf p}}] = 1 - \frac{1}{\mu}\,\left[\frac{n^2_c\,\bar{r}^2}{(1-\gamma)^2 - \bar{r}^2} + \frac{n^2_h\,\bar{r}^2}{1 - \bar{r}^2} + \frac{2\,n_c\,n_h\,\bar{r}^2}{1-\gamma - \bar{r}^2}\right],
\end{equation}
in which $\bar{r}$ is given by \eqref{eq:r_bar}.
\end{proposition}

\begin{proof}
The key ingredient for the proof lies in the following alternative formula for the Gini index \eqref{def1:Gini} (see for instance \cite{cao_gini_2024}): for a probability mass function ${\bf p} \in \mathcal{P}(\mathbb N)$ with a finite mean value $\mu \in \mathbb{R}_+$, it holds that \[G[{\bf p}] = 1 - \frac{1}{\mu}\,\sum_{n\geq 0} (1-F_n)^2,\] where $F_n = \sum_{i\leq n} p_i$ for all $i\in \mathbb N$ represents the cumulative distribution function associated to the probability distribution ${\bf p}$. Denoting by $\{F^c_n\}_{n \in \mathbb N}$ and $\{F^h_n\}_{n \in \mathbb N}$ the cumulative distributions associated to the geometric distributions $\bar{{\bf p}}^c$ and $\bar{{\bf p}}^h$ \eqref{eq:equilibria}, respectively. Since we have
\begin{equation*}
F^h_n = (1-\bar{r})\,\sum_{k\leq n} \bar{r}^k = 1-\bar{r}^{n+1} ~\textrm{and}~ F^c_n = \left(1-\frac{\bar{r}}{1-\gamma}\right)\,\sum_{k\leq n} \left(\frac{\bar{r}}{1-\gamma}\right)^k = 1-\left(\frac{\bar{r}}{1-\gamma}\right)^{n+1}
\end{equation*}
for each $n \in \mathbb N$, the advertised formula for \eqref{eq:Gini_equili} $G[\bar{{\bf p}}]$ follows immediately from the observation that $\bar{F}_n = n_c\,F^c_n + n_h\,F^h_n$, where $\{\bar{F}_n\}_{n \in \mathbb N}$ denotes the cumulative distribution associated to the distribution $\bar{{\bf p}}$.
\end{proof}

Our basic economical intuition suggests that, the Gini index of the equilibrium distribution $\bar{{\bf p}}$ might be monotone increasing as the parameter $\gamma \in [0,1)$ increases, while other model parameters $\mu \in \mathbb{R}_+$ and $n_c \in (0,1)$ (or $n_h \in (0,1)$) are kept fixed. In other words, in the non-degenerate case where both types of agents are present (i.e., $0 < n_c,n_h < 1$), the wealth inequality inherent in the distribution $\bar{{\bf p}}$ is accentuated as we increase the probability of cheating (i.e., $\gamma$) for probabilistic cheaters. Loosely speaking, we speculate that enhanced tendency of cheating for probabilistic cheaters will lead to greater wealth inequality in the equilibrium distribution. Unfortunately, although the expression \eqref{eq:Gini_equili} for $G[\bar{{\bf p}}]$ is entirely explicit (which depends on the collection of model parameters $\mu$, $n_c$ or $n_h$, and $\gamma$), due to the rather complicated dependence of $\bar{r}$ \eqref{eq:r_bar} on $\mu$, $n_h$ and $\gamma$, it is quite difficult to draw useful information from the (computations of) partial derivatives of $G[\bar{{\bf p}}]$ with respect to these model parameters. Instead, we resort to numerical evidences in support of our intuition regarding the behavior of the Gini index $G[\bar{{\bf p}}]$ as $\gamma \in [0,1)$ varies (while holding other parameters unchanged).

Figure \ref{fig:Gini_gamma} displays the dependence of $G[\bar{{\bf p}}]$ on $\gamma \in [0,1)$ for a finite collection of fixed $n_h \in \{0.2,0.4,0.6,0.8\}$ and $\mu = \{5,10\}$, where we discretized the interval $[0,1)$ into 100 equally spaced sample points ranging from 0 to 0.99. It can be readily observed that the Gini index $G[\bar{{\bf p}}]$ indeed exhibits the conjectured monotonicity as we increase $\gamma$.

\begin{figure}[!htb]
  \begin{subfigure}{0.47\textwidth}
    \centering
    \includegraphics[scale=0.6]{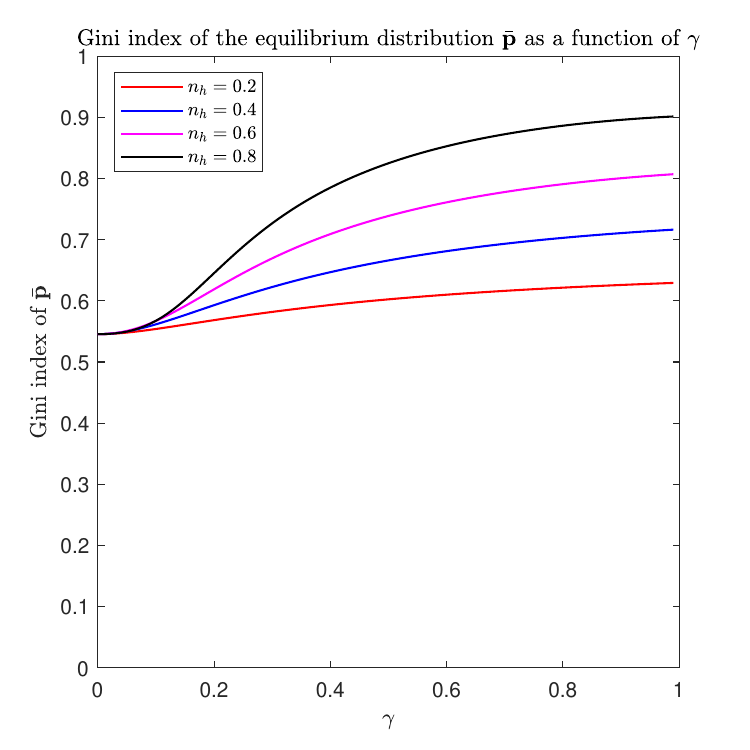}
  \end{subfigure}
  \hspace{0.1in}
  \begin{subfigure}{0.47\textwidth}
    \centering
    \includegraphics[scale=0.6]{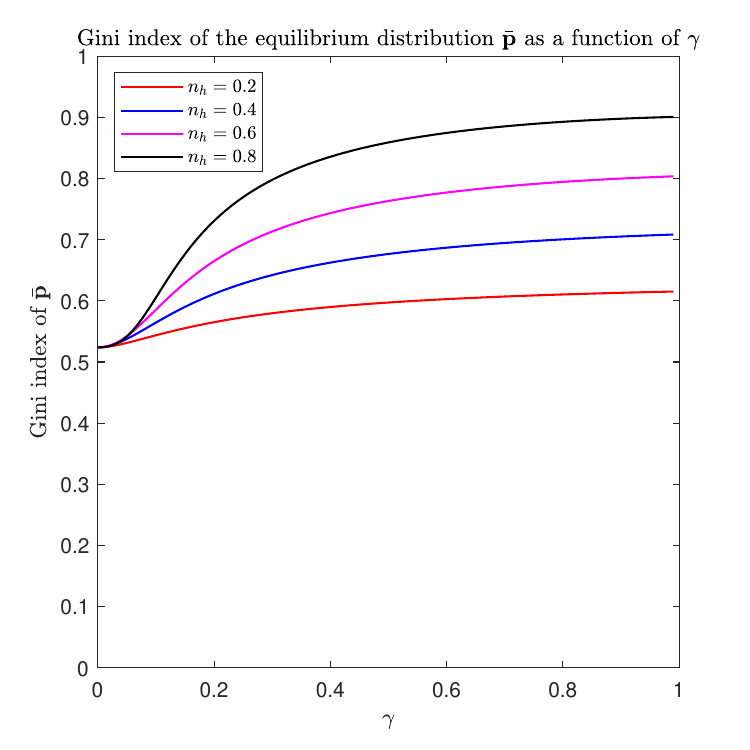}
  \end{subfigure}
  \caption{{\bf Left}: The Gini index $G[\bar{{\bf p}}]$ of the equilibrium distribution as a function of $\gamma$ for various choices of $n_h$ and $\mu = 5$. {\bf Right}: The Gini index $G[\bar{{\bf p}}]$ of the equilibrium distribution as a function of $\gamma$ for various choices of $n_h$ and $\mu = 10$. In both scenarios, the Gini index $G[\bar{{\bf p}}]$ is monotone increasing as $\gamma$ increases.}
  \label{fig:Gini_gamma}
\end{figure}

\section{Conclusion}\label{sec:5}
\setcounter{equation}{0}

In this manuscript, we investigated a variant of the classical Bennati-Dragulescu-Yakovenko (BDY) game (inspired by the very recent work \cite{blom_hallmarks_2024}) where the concept of probabilistic cheaters is introduced into the model set-up. In particular, the economic system departs from the classical framework via the introduction of heterogeneity among agents. Although the presence of another type of agents (i.e., probabilistic cheaters) leads to additional complication in the description and analysis of the agent-based model, we managed to carry out a large time analysis for the deterministic mean-field ODE system associated to the random asset-exchange model, via a discovery of a novel entropy-like functional which generalizes the usual maximum entropy principle satisfied by geometric distributions.

To the best of our of knowledge, generalizations of the basic BDY model where the fundamental assumption of indistinguishable agents is removed (and two distinct types of agents are involved) have not been studied extensively in the econophysics literature, and the model examined in this manuscript along with the one in a recent study \cite{blom_hallmarks_2024}, represents a novel and challenging direction for further research in econophysics. For instance, introducing probabilistic cheaters into other closely related variants of the BDY game (where the game might be biased towards richer or poorer agents \cite{cao_derivation_2021}) could be an interesting avenue for further exploration.

Lastly, we emphasize that the present work also leaves many intricate and intriguing questions that are worth investigating systematically in future research endeavors. To name a few, is it possible to establish a quantitative convergence guarantee for the $\HH$ functional $\HH[({\bf p}^c, {\bf p}^h)]$ ? Can we show rigorously the monotonicity of the Gini index of the equilibrium distribution $G[\bar{{\bf p}}]$ with respect to $\gamma \in [0,1)$ (with fixed $\mu > 0$ and $n_h \in (0,1)$) ? What if we assume that the likelihood of cheating $\gamma \in [0,1)$ for a given probabilistic cheater actually depend on the wealth of the agent (i.e., $\gamma = \gamma(n)$, which might be assumed to be non-decreasing if we speculate that richer agents exhibit diminished honesty) ? We believe that answers to these questions will lead to an improved understanding of the effect of introducing probabilistic cheaters and their role in accentuating or mitigating wealth inequality within our artificial economic society.


\begin{thebibliography}{99}

\bibitem{blom_hallmarks_2024}
Kristian Blom, Dmitrii E. Makarov, and Alja{\v{z}} Godec.
\newblock Hallmarks of deception in asset-exchange models.
\newblock {\em Physical Review Research}, 7(1):013279, 2025.

\bibitem{cao_bias_2023}
Fei Cao, and Stephanie Reed.
\newblock A biased dollar exchange model involving bank and debt with discontinuous equilibrium.
\newblock {\em Mathematical Modelling of Natural Phenomena}, 20:5, 2025.

\bibitem{cao_derivation_2021}
Fei Cao, and Sebastien Motsch.
\newblock Derivation of wealth distributions from biased exchange of money.
\newblock {\em Kinetic \& Related Models}, 16(5):764--794, 2023.

\bibitem{cao_explicit_2021}
Fei Cao.
\newblock Explicit decay rate for the Gini index in the repeated averaging model.
\newblock {\em Mathematical Methods in the Applied Sciences}, 46(4):3583--3596, 2023.

\bibitem{cao_fractal_2024}
Fei Cao, and Roberto Cortez.
\newblock Fractal opinions among interacting agents.
\newblock {\em arXiv preprint arXiv:2408.01624}, 2024.

\bibitem{cao_gini_2024}
Fei Cao.
\newblock From Gini index as a Lyapunov functional to convergence in Wasserstein distance.
\newblock {\em arXiv preprint arXiv:2409.15225}, 2025.

\bibitem{cao_interacting_2024}
Fei Cao, and Pierre-Emmanuel Jabin.
\newblock From interacting agents to Boltzmann-Gibbs distribution of money.
\newblock {\em Nonlinearity}, 37(12):125020, 2024.

\bibitem{cao_iterative_2024}
Fei Cao, and Stephanie Reed.
\newblock The iterative persuasion-polarization opinion dynamics and its meanfield analysis.
\newblock {\em arXiv preprint arXiv:2408.00148}, 2024.

\bibitem{cao_k_2021}
Fei Cao.
\newblock $K$-averaging agent-based model: propagation of chaos and convergence to equilibrium.
\newblock {\em Journal of Statistical Physics}, 184(2):18, 2021.

\bibitem{cao_sticky_2024}
Fei Cao, and Sebastien Motsch.
\newblock Sticky dispersion on the complete graph: a kinetic approach.
\newblock {\em arXiv preprint arXiv:2404.08868}, 2024.

\bibitem{cao_uncovering_2022}
Fei Cao, and Sebastien Motsch.
\newblock Uncovering a two-phase dynamics from a dollar exchange model with bank and debt.
\newblock {\em SIAM Journal on Applied Mathematics}, 83(5):1872--1891, 2023.

\bibitem{cao_uniform_2024}
Fei Cao, and Roberto Cortez.
\newblock Uniform propagation of chaos for a dollar exchange econophysics model.
\newblock {\em European Journal of Applied Mathematics}, 36(1):27--39, 2025.


\bibitem{chakraborti_statistical_2000}
Anirban Chakraborti, and Bikas~K. Chakrabarti.
\newblock Statistical mechanics of money: how saving propensity affects its distribution.
\newblock {\em The European Physical Journal B-Condensed Matter and Complex Systems}, 17(1):167--170, 2000.

\bibitem{chaintron_propagation_2022}
Louis-Pierre Chaintron, and Antoine Diez.
\newblock Propagation of chaos: A review of models, methods and applications. \RN{1}. Models and methods.
\newblock {\em Kinetic and Related Models}, 15(6):895--1015, 2022.

\bibitem{chaintron_propagation_2022_partII}
Louis-Pierre Chaintron, and Antoine Diez.
\newblock Propagation of chaos: A review of models, methods and applications. \RN{2}. Applications.
\newblock {\em Kinetic and Related Models}, 15(6):1017--1173, 2022.

\bibitem{chatterjee_pareto_2004}
Arnab Chatterjee,  Bikas~K. Chakrabarti, and Subhrangshu Sekhar Manna.
\newblock Pareto law in a kinetic model of market with random saving propensity.
\newblock {\em Physica A: Statistical Mechanics and its Applications}, 335(1-2):155--163, 2004.

\bibitem{cover_elements_1999}
Thomas M. Cover.
\newblock Elements of information theory.
\newblock {\em John Wiley \& Sons}, 1999.


\bibitem{degond_continuum_2017}
Pierre Degond, Jian-Guo Liu, Sara Merino-Aceituno, and Thomas Tardiveau.
\newblock Continuum dynamics of the intention field under weakly cohesive social interaction.
\newblock {\em Mathematical Models and Methods in Applied Sciences}, 27(01):159--182, 2017.

\bibitem{dragulescu_statistical_2000}
Adrian Dragulescu, and Victor~M. Yakovenko.
\newblock Statistical mechanics of money.
\newblock {\em The European Physical Journal B-Condensed Matter and Complex
  Systems}, 17(4):723--729, 2000.

\bibitem{during_kinetic_2008}
Bertram D{\"u}ring, Daniel Matthes, and Giuseppe Toscani.
\newblock Kinetic equations modelling wealth redistribution: a comparison of approaches.
\newblock {\em Physical Review E}, 78(5):056103, 2008.


\bibitem{jabin_clustering_2014}
Pierre-Emmanuel Jabin, and Sebastien Motsch.
\newblock Clustering and asymptotic behavior in opinion formation.
\newblock {\em Journal of Differential Equations}, 257(11):4165--4187, 2014.

\bibitem{lanchier_rigorous_2017}
Nicolas Lanchier.
\newblock Rigorous proof of the Boltzmann–Gibbs distribution of money on connected graphs.
\newblock {\em Journal of Statistical Physics}, 167(1):160--172, 2017.

\bibitem{lanchier_rigorous_2019}
Nicolas Lanchier, and Stephanie Reed.
\newblock Rigorous results for the distribution of money on connected graphs (models with debts).
\newblock {\em Journal of Statistical Physics}, 176(5):1115--1137, 2019.

\bibitem{lanchier_distribution_2022}
Nicolas Lanchier, and Stephanie Reed.
\newblock Distribution of money on connected graphs with multiple banks.
\newblock {\em arXiv preprint arXiv:2201.11930}, 2022.

\bibitem{loy_essentials_2025}
Nadia Loy, and Andrea Tosin.
\newblock Essentials of the kinetic theory of multi-agent systems.
\newblock {\em arXiv preprint arXiv:2503.11554}, 2025.

\bibitem{matthes_steady_2008}
Daniel Matthes, and Giuseppe Toscani.
\newblock On steady distributions of kinetic models of conservative economies.
\newblock {\em Journal of Statistical Physics}, 130:1087--1117, 2008.

\bibitem{merle_cutoff_2019}
Mathieu Merle, and Justin Salez.
\newblock Cutoff for the mean-field zero-range process.
\newblock {\em Annals of Probability}, 47(5):3170--3201, 2019.
\newblock Publisher: Institute of Mathematical Statistics.

\bibitem{miao_nonequilibrium_2023}
Maggie Miao, Kristian Blom, and Dmitrii Makarov.
\newblock Nonequilibrium statistical mechanics of money/energy exchange models.
\newblock {\em Journal of Physics A: Mathematical and Theoretical}, 57(15):155003, 2024.

\bibitem{pareschi_interacting_2013}
Lorenzo Pareschi, and Giuseppe Toscani.
\newblock Interacting multiagent systems: kinetic equations and Monte Carlo methods.
\newblock {\em OUP Oxford}, 2013.

\bibitem{pereira_econophysics_2017}
Eder Johnson de Area~Leão Pereira, Marcus~Fernandes da~Silva, and HB~de~B.
  Pereira.
\newblock Econophysics: {Past} and present.
\newblock {\em Physica A: Statistical Mechanics and its Applications},
  473:251--261, 2017.

\bibitem{savoiu_econophysics_2013}
Gheorghe Savoiu.
\newblock Econophysics: Background and Applications in Economics, Finance, and Sociophysics.
\newblock {\em Academic Press}, 2013.

\bibitem{sznitman_topics_1991}
Alain-Sol Sznitman.
\newblock Topics in propagation of chaos.
\newblock In {\em Ecole d'été de probabilités de {Saint}-{Flour}
  {XIX}—1989}, pages 165--251. Springer, 1991.


\bibitem{villani_review_2002}
C{\'e}dric Villani.
\newblock A review of mathematical topics in collisional kinetic theory.
\newblock {\em Handbook of mathematical fluid dynamics}, 1(71-305):3--8, 2002.

\bibitem{villani_entropy_2006}
C{\'e}dric Villani.
\newblock Entropy production and convergence to equilibrium for the Boltzmann equation.
\newblock In {\em XIVth International congress on mathematical physics}, pages 130--144. World Scientific, 2006.


\end{thebibliography}
\end{document}